\documentclass[12pt]{article}
\usepackage{amsmath,amscd,amsthm,amsfonts,amssymb}
\begin{document}
\newtheorem{defn}{Definition}
\newtheorem{fact}{Fact}
\newtheorem{thm}{Theorem}
\newtheorem{prop}{Proposition}
\newtheorem{lem}{Lemma}
\newtheorem{conj}{Conjecture}
\newtheorem{slem}{Sublemma}
\newtheorem{cor}{Corollary}
\newtheorem{ass}{Assumption}
\newtheorem{ques}{Question}
\newtheorem*{mlem*}{Main Lemma}

\newcounter{constant} 
\newcommand{\newconstant}[2]{\refstepcounter{constant}\label{#1}}
\newcommand{\useconstant}[2]{C_{\ref{#1}}}
\title{Random walks and the symplectic representation of the braid groups}
\author{Marc Soret and Marina Ville}
\maketitle
\begin{abstract} We consider the symplectic representation $\rho_n$ of a braid group $B(n)$ in $Sp(2l,\mathbb{Z})$, for $l=\Big[\dfrac{n-1}{2}\Big]$. If $P$ is a  polynomial on the $4l^2$ coefficients of the matrices in $Sp(2l,\mathbb{Z})$, we show that the set $\{\beta\in B(n): P(\rho_n(\beta))=0\}$ is transient for non degenerate random walks on $B(n)$. \\
	We derive that the $n$-braids $\beta$ which close into a loop $\hat{\beta}$ with  $0<|det({\hat{\beta}})|\leq C$ for some constant $C$ form a transient set. And given a prime number $p$, we show that the probability for a given braid to close in a $p$-colorable loop is greater than $\dfrac{1}{p}$. \\
	We also derive that for a random $3$-braid,  the quasipositive links $(\beta\sigma_i\beta^{-1}\sigma_j)^p$ have zero signature for every integer $p$ and $1\leq i,j\leq 2$. \\ As an example of such braids, we investigate the signature of the  Lissajous toric knots $3$-braids.

\end{abstract}
\section{Preliminaries}
\subsection{Representations of the braid group}\label{paragraphe sur les representations}
If $B(n+1)$ is the group of braids with $n+1$ strands, we denote by 
$\mathcal B_t$ its {\it reduced Burau representation} which represents braids as $n\times n$ matrices with values in $\mathbb{Z}[t,t^{-1}]$. The specialization of $\mathcal B_t$ for $t=-1$, denoted $\mathcal B_{-1}$, is called the {\it integral reduced Burau representation}. The representation $\mathcal B_{-1}$ is closely related to the {\it symplectic representation} of $B(n+1)$ which we now recall, using the description of A'Campo ([A'C] \S  1).\\
We consider a surface $X$ such that $H_1(X,\mathbb{Z})$ is generated by $n$ loops; the generators of $B(n+1)$ act on $H_1(X,\mathbb{Z})$ as Dehn twists w.r.t. corresponding generators of $H_1(X,\mathbb{Z})$. This action is conjugate to the integral reduced Burau representation ${\mathcal B}_{-1}$. \\  
This action of $B(n+1)$ preserves the intersection form $I$ on $H_1(X,\mathbb{Z})$. 
\begin{enumerate}
	\item 
	If $n$ is even, $n=2l$, the form $I$ is symplectic nondegenerate on $H_1(X,\mathbb{Z})$.
	\item
	If $n$ is odd, $n=2l+1$, the kernel of $I$ is a line $K$; the map $\mathcal B_{-1}$ restricts to the identity on $K$ and  acts symplectically on the quotient $H_1(X,\mathbb{Z})/K$ 
\end{enumerate}
Thus, in both cases, we get  a {\it symplectic representation}
\begin{equation}\label{representation symplectique0}
\rho_n:B(n+1)\longrightarrow Sp(2l,\mathbb{Z})
\end{equation} 
\subsection{Random walks on groups}\label{paragraph on random walks}
If $\mu_1,\mu_2$ are probability measures on a discrete group $G$ with finite support, their {\it convolution} is defined for $g\in G$ as 
\begin{equation}\label{definition de la convolution}
(\mu_1\star\mu_2)(g)=\sum_{h\in G}\mu_1(gh^{-1})\mu_2(h)
\end{equation}
\begin{defn}\label{non degenerate}
	A  probability $\mu$ on a discrete group $G$ is nondegenerate if $supp(\mu)$ generates $G$ as a semigroup.\end{defn} Malyutin ([Mal]) defines the {\it  right random $\mu$-walk} as the Markov chain which starts at the identity of $G$ and with transition probability $P(g,h)=\mu(gh^{-1})$. If $X$ is a subset of $G$, the probability that this walk hits $X$ at the $k$-th step is $\mu^{\star k}(X)$. 
\subsection{Signature and determinant of a link}\label{definition de la signature}
These are classical invariants of links and knots. 
If $\gamma$ is a link on $\mathbb{S}^3$ which bounds a Seifert surface $\Sigma$ in $\mathbb{S}^3$, we consider the bilinear form called {\it Seifert form} $$\Phi: H_1(\Sigma,\mathbb{Z})\times H_1(\Sigma, \mathbb{Z})\longrightarrow\mathbb{Z}$$
$$(\alpha,\beta)\mapsto lk(\hat{\alpha},\beta)$$ where $lk$ is the linking number and $\hat{\alpha}$ is the link obtained by pushing $\alpha$ in the direction normal to $\Sigma$ in $\mathbb{S}^3$. Given a base for $H_1(\Sigma, \mathbb{Z})$, we define the matrix $V$ of $\Phi$ and consider the bilinear form defined by the matrix $V+^tV$: its determinant is the {\it determinant} of the link and its {\it signature} is the signature of the link. We point out that the literature contains two different definitions of the signature which coincide up to sign.\\
If $\Delta_L$ is the Alexander polynomial of a link, then $det(L)=\Delta_L(-1)$.\\ 
We recall (cf. [Mu]) that for a knot $K$, 
\begin{equation}
|sign(K)|\leq 2g_4(K)
\end{equation} where $g_4$ denotes the topological $4$-genus.  Thus a slice knot has zero signature but the converse is not true and we will see examples of that below.
\subsection{Positive links and quasipositive braids}
A {\it positive link} is a link which has a diagram with all positive crossings.
\begin{thm}([Pr])
	A positive link has strictly negative signature.
\end{thm} 
So it is interesting to look at the signature of quasipositive braids:
\begin{defn} ([Ru])
	A braid $\beta$ is {\it quasipositive} if it is a product of conjugates of positive braid generators
	\begin{equation}\label{expression de la tresse quasipositive}
	\beta=\prod_{i=1}^k\gamma_{i}\sigma_{j_i}\gamma_{i}^{-1}
	\end{equation}
\end{defn}
If we close the  quasipositive $N$-braid (\ref{expression de la tresse quasipositive}) in a link $\hat{\beta}$, Rudolph proved ([Ru]) that
\begin{equation}\label{quasipositif.link}
\chi_4(\hat{\beta})=N-k
\end{equation}
where $\chi_4$ denotes the largest Euler characteristic of a {\it smooth} surface in $\mathbb{B}^4$ bounded by $\hat{\beta}$.\\
Tanaka ([Ta]) constructed examples of quasipositive braids with zero signature and we will give some more here.

\subsection{Transient sets in the braid groups from quasimorphisms}
 We recall that 
a {\it quasimorphism} on a group $G$ is a function $\Phi:G\longrightarrow \mathbb{R}$ such that there exists a positive constant $D_\Phi$ with
\begin{equation}\label{quasimorphisme}
\forall a,b\in G,\ |\Phi(ab)-\Phi(a)-\Phi(b)|\leq D_\Phi
\end{equation}
where $D_\Phi$ is called the {\it defect} of $\Psi$.\\
Quasimorphisms give exemples of transient sets in the braid groups:
\begin{thm}\label{malyutin}
	([Mal]) Let $G$ be a countable group and $\mu$ a nondegerate probability on $G$ (see Definition \ref{non degenerate}). If  a subset $S$ of $G$ has bounded image under an unbounded quasimorphism $\Phi$, then the probability that the right random $\mu$-walk hits $S$ at the $k$-th step tends to zero as $k$ tends to infinity.
\end{thm}
We will see below how Gambodau-Ghys's formula ([G-G]) for the signature turns it into a quasimorphism.  Several other link invariants such as $\chi_4$ can be used to define   quasimorphisms on the braid group ([Br], [B-K]).\\
Other exemples of transient sets on the braid groups have been constructed by Ito ([It]).

\section*{Acknowledgements}
We thank Tahl Nowik and Chaim Even-Zohar for reading an earlier draft of this paper, Misha Brandenbursky for very helpful conversation on braids and quasimorphisms and Florence Lecomte for necessary advice on algebraic geometry over finite fields.

\section{The results}
\begin{ass}\label{assump}  The random walks that we consider here are right random $\mu$ walks for a nondegenerate probability $\mu$ on $B(n+1)$. \\
For $B(3)$, we assume moreover that $\mu(\sigma_i)\neq 0$, $\mu(\sigma_i^{-1})\neq 0$ for $i=1,2$.
\end{ass}
\begin{thm}\label{prop sur les transients}
	Let $P$ be a polynomial in $(2l)^2$ variables with coefficients in $\mathbb{Z}$. If $M=(m_{ij})\in Sp(2l,\mathbb{Z})$, we let 
	$P(M)=P(m_{11},...,m_{12l},m_{21},...,m_{22l},m_{2l1},...,m_{2l2l}).$\\
	Suppose that $P$ does not vanish identically on $Sp(2l,\mathbb{Z})$. Then the set
	$$\{\beta\in B(n+1): P(\rho_n(\beta))=0\}$$
	is transient for the right random $\mu$-walk.
\end{thm}
\subsection{The determinant}\label{resultats determinant}
We will use the expression of the Alexander polynomial in terms of the Burau representation to derive from Theorem \ref{prop sur les transients} the following.

\begin{thm} \label{thm determinant} Let $n$ be an integer and $C$ a real number.
	\begin{enumerate}
		\item
		If $n+1$ is odd, $\{\beta\in B(n+1): |det(\hat{\beta})|<C\}$ is transient
		\item
		If $n+1$ is even, 
		\begin{enumerate}
			\item $\{\beta\in B(n+1): 0<|det(\hat{\beta})|<C\}$ is transient			
			\item $\{\beta\in B(n+1): \beta\ \mbox{closes into a knot and }\ |det(\hat{\beta})|<C\}$ is transient
		\end{enumerate}	
	\end{enumerate}	
\end{thm}
We recall that, given a prime number $p$, a link $L$ is $p$-{\it colorable} if $det(L)$ is divisible by $p$. We derive
\begin{thm} \label{thm color} Let $n\geq 3, k$ be positive integers and let $p\geq 3$ a prime number. Let $P_p(n,k)$ the probability that a $n$-braid obtained by $k$ steps of the right random $\mu$-walk closes in a $p$-colorable link and we let $P_p(n)$ the limit of $P_p(n,k)$ when $k$ tends to infinity.\\
	1) For all $n\geq 3$, $P_p(n)\geq  \dfrac{p}{p^2-1}+\dfrac{17}{40p^2}$\\
	2) If $n$ is odd, 	$P_p(n)\leq  \dfrac{p}{p^2-1}+\dfrac{9}{5p^2}$
\end{thm}
\subsection{The signature of $3$-braids}
Ghys-Gambodau ([G-G]) use the integral Bureau representation to give a formula for the signature of links given by closures of braids. We derive
\begin{prop}\label{thm sur les random walks}
	For a random walk as in Assumption \ref{assump}, the set
	\begin{equation}\label{definition de l'ensemble S}
	S=\{\beta\in B(3): \exists n\in\mathbb{N}, \ \exists i,j\in\{1,2\}\  \mbox{such that}\ \ \  
	sign\big((\beta\sigma_i\beta^{-1}\sigma_j)^n\big)\neq 0\}
	\end{equation}
		is transient. 
\end{prop}
The $(\beta\sigma_i\beta^{-1}\sigma_j)^n$'s are quasipositive, thus
 (\ref{quasipositif.link})  tells us that for every $n$, $\chi_4 \big((\beta\sigma_1\beta^{-1}\sigma_1)^n\big)=3-2n$; in particular, if $(\beta\sigma_1\beta^{-1}\sigma_1)^n$ closes into a knot, $g_4((\beta\sigma_1\beta^{-1}\sigma_1)^n)=n-1$.\\
\\
In \S \ref{rappels sur les lissajous toriques}, we recall the Lissajous toric knots with $3$ strands (cf. [S-V]); they are naturally of the form $(\beta\sigma_2\beta^{-1}\sigma_1^{\pm 1})^n$ and we prove
\begin{prop}\label{signature de lissajous}
	A Lissajous toric knot $K(3,p,q)$ has zero signature unless it is isotopic to a torus knot. 
\end{prop} 
We conclude by statistical estimates, done with Sagemath.\\
\\
These results suggest  possible generalizations.
\begin{ques}
	Let $n\in\mathbb{N}$ and let $\mathcal I$ be a numerical link invariant which is unbounded on the closures of the links represented by a $n$-braid. If $C$ is a constant, is the set $\{{\mathcal I}(\hat{\beta})\leq C\}$  transient for the right random $\mu$-walks?
\end{ques}

  \begin{ques}For an integer $m$, we set $$A_m^{(1)}=\prod_{1\leq 2i+1\leq m}\sigma_{2i+1}\ \ \ \ A_m^{(2)}=\prod_{1\leq 2i\leq m}\sigma_{2i}$$
  	Is the following set $S_m$ transient in $B(m+1)$ for the right random $\mu$-walks?
  	$$S_m=\{\beta\in B(m+1)\slash \exists n\in\mathbb{N}, \exists i,j\in\{1,2\}\ \  \mbox{such that}\ \ \  
  	sign\big((\beta A_m^{(i)}\beta^{-1} A_m^{(j)})^n\big)\neq 0\}$$
 \end{ques}

\section{Random walks: proof of Theorem \ref{prop sur les transients}}
Similarly to Rivin in [Ri], we introduce the finite groups $Sp(2l,\mathbb{Z}_p)$'s for a prime number $p\geq 3$.\\
A theorem of A'Campo  ([A'C], Theorem 1 (1)) shows that the representation $B(n)\longrightarrow Sp(2l,\mathbb{F}_p)$  is surjective. So we consider the surjective map
\begin{equation}\label{representation symplectique}
\Pi_p:B(n)\overset{\rho_n}{\longrightarrow} Sp(2l,\mathbb{Z})
\longrightarrow Sp(2l,\mathbb{F}_p)
\end{equation}   
\begin{lem}\label{uniforme sur Psp} 
	Let $p\geq 3$, Consider the image probability $\Pi_p\mu$ of $\mu$ via $\Pi_p$ on $Sp(2l,\mathbb{F}_p)$, i.e. $\Pi_p(X)=\mu(\Pi_p^{-1}(X))$.  If $m$ tends to infinity, $(\Pi_p\mu)^{\star m}$ converges to the equidistributed measure on $Sp(2l,{\mathbb{F}_p})$. 
\end{lem}
\begin{proof}
The lemma follows from the following two theorems. 
\begin{thm}\label{groupe fini}
	(see for exemple [Di])
	Let  $G$ be a finite group and $\mu$ a probability measure on $G$ such that 
	\begin{enumerate}
		\item 
$G$ is generated as a semigroup by 	$supp(\mu)$
\item $supp(\mu)$ is not included in a coset of $G$ by a non trivial normal subgroup $H$ of $G$.
\end{enumerate}
  Then $\mu^{\star m}$ converges to the equidistributed measure on $G$ as $m$ tends to infinity.
\end{thm}
\begin{thm}\label{psl simple} ([As]) If $n> 2$ or $p> 3$, the group $PSp(2n,{\mathbb{F}_p})$ is simple.
	\end{thm} 
If $H$ is a non trivial normal subgroup of $Sp(2n,\mathbb{F}_p)$, 
Theorem \ref{psl simple} tells us that it maps into $\{Id\}$, hence $H$ is included into $\{kId: k\in \mathbb{F}_p\}$. Thus $supp(\mu)$ cannot be a  coset $gH$ and Theorem \ref{groupe fini} applies.\\
We now show that if $n=2$ and $p=3$ we can also apply Theorem \ref{groupe fini}. We recall that the only non trivial normal subgroup of $PSL(2,\mathbb{F}_3)$ is the Klein group  so the only normal subgroups of $SL(2,\mathbb{Z})$ are 
\begin{itemize}
	\item the center $\{\pm Id\}$
	\item 
	$H=\{\pm Id, \pm\begin{pmatrix}
	0 & 1 \\
	-1 & 0 
	\end{pmatrix}, \pm\begin{pmatrix}
	-1 & -1 \\
	-1 & 1 
	\end{pmatrix}, \pm\begin{pmatrix}
	-1 & 1 \\
	1 & 1 
	\end{pmatrix}\}$
	\end{itemize}
Now for the reduced Burau representation ${\mathcal B}_{-1}:B(3)\longrightarrow SL(2,\mathbb{Z})$, we have
\begin{equation}\label{les generateurs}
s_1={\mathcal B}_{-1}(\sigma_1)=\begin{pmatrix}
1 & 0 \\
-1 & 1 
\end{pmatrix}\ \ \ \ s_2={\mathcal B}_{-1}(\sigma_2)=\begin{pmatrix}
1 & 1 \\
0 & 1 
\end{pmatrix}
\end{equation}
The support of $\Pi_p\mu$ contains the matrices $s_1$ and $s_2^{-1}$. Suppose there exists $g\in SL(2,\mathbb{F}_3)$ such that $s_1$ and $s_2^{-1}$ belong to $gH$; then $s_2s^{-1}=ghg^{-1}$ for some $h\in H$, hence $(s_2s_1^{-1})^4=Id$, which is not true. Thus the assumptions of Theorem \ref{groupe fini} is verified.
\end{proof}
We now investigate the density of the zero set in $Sp(2l, \mathbb{F}_p)$ of the polynomial  $P$ appearing in Theorem \ref{prop sur les transients}.\\
We write $P$ in terms of all the multi-indices, $P(M)=\sum a_I m^I$, where the $a_I$'s belong to $\mathbb{Z}$. If $p$ is an integer, we denote by $P_p$ its reduction modulo $p$. Since $P$ is not identically zero, we derive an infinite set $\mathcal P$ of prime numbers $p$ such that $P_p$ is not identically zero.  To prove the theorem, we need to estimate
\begin{equation}
\frac{|P_p^{-1}(0)|}{|Sp(2l, \mathbb{F}_p)|}
\end{equation}
We recall (see for example [O'M] or [wi]) that 
\begin{equation}\label{cardinal de Sp}
|Sp(2l,\mathbb{F}_p)|=\prod_{m=1}^l\big[(p^{2m}-1)p^{2m-1}\big]\geq \frac{1}{2^{l}}\prod_{s=1}^{2l}p^{s}
= \dfrac{p^{2l^2+l}}{2^{l}}
\end{equation}
To estimate $|P_p^{-1}(0)|$, we use the following result by Lachaud and Rolland 
\begin{thm}
	([L-R]) Let $K$ be the algebraic closure of $\mathbb{F}_p$. Let $X$ be an algebraic variety of $K^n$ of dimension $m$ which is the zero set of a family of polynomials $(f_1,...,f_r)$. Let $d_i=deg(f_i)$. Then 
	$$|X\cap \mathbb{F}_p^n|\leq d_1...d_rp^m$$
\end{thm}
We view $P_p$ as a polynomial with coefficients in $K$ and let 
\begin{equation}
X=Sp(2l,K)\cap P_p^{-1}(0)\subset \mathbb{F}_p^{4l^2}
\end{equation} We know that for a field $K$, $Sp(2l,K)$ is irreducible as an algebraic variety ([Mi]), hence $X$ is of dimension strictly smaller that $Sp(2l,K)$, i.e. $dim(X)\leq 2l^2+l-1$.  Thus, for some constant $C_1(2l,P)$ depending on $2l$ and on the degree of $P$,
\begin{equation}\label{application de lachaud-rolland}
|X\cap \mathbb{F}_p^{4l^2}|\leq C_1(2l,P) p^{2l^2+l-1}
\end{equation} 
Putting (\ref{cardinal de Sp}) and (\ref{application de lachaud-rolland}) together, we derive that  
\begin{equation}\label{estimee de conclusion}
\dfrac{|X|}{|Sp(2l, \mathbb{F}_p)|}\leq \dfrac{C_2(2l,P)}{p}
\end{equation}
for some constant $C_2(2l,P)$. Thus, given a $\epsilon>0$, we pick a prime number $p$ in $\mathcal P$ such that 
\begin{equation}
2\dfrac{C_2(2l,P)}{p}<\epsilon
\end{equation}
There exists an integer $k_0$ such that for every $k>k_0$, we have $$\mu^{\star k}(P^{-1}(0))\leq (\Pi_p\mu)^{\star k}(P_p^{-1}(0))<\epsilon.$$ 
\qed
\section{The determinant} 
We recall that the Alexander polynomial of the closure of a braid $\beta\in B(n+1)$ can be expressed in terms of the Burau representation $\mathcal B$ ([K-T]),
\begin{equation}\label{burau et alexander}
\Delta_{\hat{\beta}}(t)=\frac{1-t}{1-t^{n+1}}det({\mathcal B}_{-t}(\beta)-I)
\end{equation}
We derive the proof of Theorems \ref{thm determinant} and \ref{thm color}.
\subsection{Proof of Theorem \ref{thm determinant}}\label{proof determinant}
	Case 1, where $n+1$ is odd, follows immediately from Theorem. \ref{prop sur les transients}.\\
	So we assume that $n+1$ is even (case 2). For $\beta\in B(n+1)$, we let $P_\beta$ the characteristic polynomial of $\mathcal B_{-1}(\beta)$. Since  $\mathcal B_{-1}(\beta)(X)=X$, we have $P_\beta(1)=0$.
	\begin{lem}
		There exist $\beta\in B(n+1)$ such that $P'(1)\neq 0$.
	\end{lem}
\begin{proof}
A'Campo ([A'C]) proved that $\rho(B(n+1))$ contains the congruence subgroup $2$ of $Sp(2l,\mathbb{Z})$. Thus there is a braid $\beta$ with $\rho(\beta)$ is the matrix given by the blocks $\begin{pmatrix}-1&0\\2&-1\end{pmatrix}$ and its characteristic polynomial is $(1+X)^{\frac{n+1}{2}}$ thus $P_\beta(X)=(X-1)(1+X)^{\frac{n+1}{2}}$.\end{proof}
It follows that $\beta\mapsto P_\beta'(1)$ is not identically zero. Thus Theorem \ref{prop sur les transients} tells us that the set $Z$ of $n$-braids $\beta$ such that $P_\beta(X)$ does not have a factor with multiplicity greater than $1$ is recurrent. \\
For each braid $\beta$ in $Z$, there is a subspace $H_\beta$ stable by ${\mathcal B}_{-1}(\beta)$ with $\mathbb{R}^{n}=\mathbb{R}K\oplus H_\beta$, so we can complete $K$ into a basis $e_2,...,e_{n}$ such that the matrix $M(\beta))$ of ${\mathcal B}_{-1}(\beta)$ verifies 
	\begin{itemize}
		\item
		$M_{-1}(\beta)_{11}=1$
		\item if $i\neq 1, M_{-1}(\beta)_{1i}=M_{-1}(\beta)_{i1}=0$
	\end{itemize}
	The matrix $A=(M_{-1}(\beta))_{i,j\geq 2}$ is the matrix of $\rho_n(\beta)$ in the basis $\bar{e}_2,...,\bar{e}_{n}$ where $\bar{e}_i$ is the image of $e_1$ in $\mathbb{R}^{n}/\mathbb{R}K$.\\
	We let $M_{-1+s}(\beta)$ be the matrix of ${\mathcal B}_{-1+s}$ in the basis $(K,e_2,...,e_{n})$. Its coefficients are rational fractions with integer coefficients in $s$. Hence, for each $i,j$, there exists an integer $a_{ij}$ such that
	\begin{equation}
	M_{-1+s}(\beta)_{ij}=M_{-1}(\beta)_{ij}+a_{ij}s+{\mathcal O}(s^2)
	\end{equation}
	We compute $det(M_{-1+s}(\beta)-I)$ by developing w.r.t. the first column or first line and get
	\begin{equation}
	det(M_{-1+s}(\beta)-I)=sa_{11}det(\rho_n(\beta)-Id)+{\mathcal O}(s^2)
	\end{equation}
	\begin{equation}\label{determinant impair}
	det(\hat{\beta})=\Delta_{\hat{\beta}}(-1)=\lim_{s\longrightarrow 0}\Delta_{\hat{\beta}}(-1+s)=\dfrac{2}{n+1}a_{11}det(\rho_n-Id)
	\end{equation}
	Thus, if $0<|det(\hat{\beta})|\leq C$, $a_{11}\neq 0$ and  
	$$|det(\rho_n(\beta)-Id)|<|\dfrac{C(n+1)}{2a_{11}}|\leq |C(n+1)|$$
	so Theorem \ref{prop sur les transients} applies and this proves 2 (a). To prove (b), we recall
	\begin{itemize}
		\item 
		the determinant of a knot is never zero ([Ro])
		\item
		the probability of a $(n+1)$-braid closing to a knot is asymptotically $\dfrac{1}{n+1}$ ([Ma]) 
	\end{itemize}

\subsection{Proof of Theorem \ref{thm color}}\label{proof colorable} 
A $(n+1)$-braid closes in a  $p$-colorable link iff $p$ divides $\Delta_{-1}(\beta)$ (\S \ref{definition de la signature} and \S \ref{resultats determinant}). \begin{enumerate}
	\item 
 If $n+1$ is odd, (\ref{burau et alexander}) tells us that $\hat{\beta}$ is $p$-colorable if and only if $p$ divides $det(\rho(\beta)-Id)$. 
 \item If $n+1$ is even, we do not have an equivalence but derive from (\ref{determinant impair}): if $p$ divides $det(\rho(\beta)-Id)$ then $\hat{\beta}$ is $p$-colorable. 
\end{enumerate}
\begin{lem}\label{estimee de tn}
	We let
\begin{itemize}
	\item $T(l,p)=\#\{M\in Sp(2l, \mathbb{F}_p): 1\ \mbox{is an eigenvalue of }\ M\}$
	\item $t_l=\dfrac{T(l,p)}{\# Sp(2l, \mathbb{F}_p)}$ if $n>0$ and $t_0=0$
\end{itemize}
Then $\dfrac{p}{p^2-1}+\dfrac{17}{40p^2}\leq t_l\leq \dfrac{p}{p^2-1}+\dfrac{9}{5p^2}$
\end{lem}
\begin{proof}
	The proof is sketched in [A-H].  We recall
	\begin{prop}([Ac]) Let $S(l,p)=\{M\in Sp(2l, \mathbb{F}_p): M\ \mbox{is unipotent}\}$ Then $\# S(l,p)=p^{2l^2}$
\end{prop}
We count the elements of $T(l,p)$ according to the dimensions of the generalized eigenspaces of $1$ and get
$$\# T(l,p)=\sum\limits_{\substack{1\leq r\leq l\\ 
		r+s=l}}\dfrac{\# Sp(2l, \mathbb{F}_p)}{\# Sp(2r, \mathbb{F}_p)\# Sp(2s, \mathbb{F}_p)}S(r,p)(\# Sp(2s, \mathbb{F}_p)-T(s,p))$$
\begin{equation}\label{formule pour tn}
t_l=\sum\limits_{\substack{1\leq r\leq l\\ 
		r+s=l}}\dfrac{p^{2r^2}}{\# Sp(2r, \mathbb{F}_p)}(1- t_s)
=\dfrac{p}{p^2-1}+\sum\limits_{\substack{2\leq r\leq l \\ 
		r+s=n}}\dfrac{1}{p^r\prod\limits_{m=1}^r\Big(1-p^{-2m}\Big)}(1- t_s)
\end{equation}
We easily prove that, if $r\geq 2$ and $p\geq 3$, $ln\Big(\dfrac{1}{1-p^{-2m}}\Big)\leq \dfrac{9}{8}p^{-2m}$\\
\[\mbox{thus}\ ln\Big(\prod\limits_{m=1}^r\dfrac{1}{1-p^{-2m}}\Big)
\leq \dfrac{9}{8}\sum_{m=1}^rp^{-2m}=\dfrac{9}{8}\Big(\dfrac{1-p^{-2r}}{p^2-1}\Big)\leq \dfrac{9}{64}\ \mbox{and (\ref{formule pour tn}) yields}\] 
\begin{equation}\label{majore tn}
t_l\leq \dfrac{p}{p^2-1}+\dfrac{1}{p^2}e^{\frac{9}{64}}\sum_{a=0}^{l-2}\dfrac{1}{p^a}\leq 
\dfrac{p}{p^2-1}+\dfrac{3}{2p^2}e^{\frac{9}{64}}\leq \dfrac{p}{p^2-1}+\dfrac{9}{5p^2}
\end{equation}
We plug (\ref{majore tn}) into (\ref{formule pour tn}) and derive
$$
t_l\geq \dfrac{p}{p^2-1}+\dfrac{1}{p^2}\underbrace{(1-\dfrac{p}{p^2-1}-\dfrac{9}{5p^2})}_{\mbox{ minimum at p=3}}\geq \dfrac{p}{p^2-1}+\dfrac{17}{40p^2} 
$$
\end{proof}
The  theorem follows from Lemma \ref{estimee de tn} and the discussion at the beginning of \S \ref{proof colorable}.
\section{The signature of $3$-braids} 
\subsection{Preliminaries: the Gambaudo-Ghys formula for the signature of a $3$-braid ([G-G])}\label{section: formule de gg} 
We use the generators $s_1$, $s_2$ of $SL(2,\mathbb{Z})$ and the expression in (\ref{les generateurs}) for $\mathcal B_{-1}$. 
\subsubsection{The Meyer cocycle and the formula for the signature}\label{meyer}
\label{section sur le meyer cocycle}
For two $3$-braids, $\alpha$ and $\beta$, [G-G] proves  
\begin{equation}\label{formule de gg}
sign(\widehat{\alpha.\beta})=
sign(\hat{\alpha})+sign(\hat{\beta})-
Meyer({\mathcal B}_{-1}(\alpha),{\mathcal B}_{-1}(\beta))
\end{equation}
where $Meyer\in H^2(SL(2,\mathbb{R}))$ is the Meyer cocycle defined as follows.
If $\gamma_1,\gamma_2\in SL(2,\mathbb{Z})$, we let
\begin{equation}\label{e.gamma.un.gamma.deux}
E_{\gamma_1,\gamma_2}=Im(\gamma_1^{-1}-Id)\cap Im(\gamma_2-Id)
\end{equation}
For a vector $e$ in $E_{\gamma_1,\gamma_2}$, we take $v_1,v_2$ in $\mathbb{R}^2$ such that 
$$e=\gamma_1^{-1}(v_1)-v_1=v_2-\gamma_2(v_2)$$
and define the quadratic form 
\begin{equation}\label{e.gamma.un.gamma.deux}
q_{\gamma_1,\gamma_2}=\Omega(e,v_1+v_2)
\end{equation}
where $\Omega$ is the standard symplectic form on $\mathbb{R}^2$.
Then $Meyer (\gamma_1,\gamma_2)$ is the signature of $q_{\gamma_1,\gamma_2}$.

\begin{fact}\label{meyer pour hyperboliques}
	If $\gamma\in SL(2,\mathbb{Z})$ is hyperbolic (i.e. $|tr(\gamma)|>2$), then for two positive integers $a,b$, $$Meyer(\gamma^a,\gamma^b)=0$$
\end{fact}
REMARK. Since a generic element of $SL(2,\mathbb{Z})$ is hyperbolic, it follows from Fact \ref{meyer pour hyperboliques} that for almost every braid $\beta\in B(3)$, $sign(\beta^n)=nsign(\beta)$.

\subsection{Proof of Proposition \ref{thm sur les random walks}}

	\begin{lem}\label{qui est hyperbolique}
		Let $M=\begin{pmatrix}
		a & b \\
		c & d 
		\end{pmatrix}\in SL(2,\mathbb{Z})$. The matrix $Ms_1M²^{-1}s_1$ (resp. $Ms_2M²^{-1}s_2$, $Ms_1M²^{-1}s_2$, $Ms_2M²^{-1}s_1$)  is hyperbolic unless 
		\begin{equation}\label{condition pour l'hyperbolicite}
		b^2\leq 4\ \ (\mbox{resp}.\ c^2\leq 4,\  d^2\leq 4,\  a^2\leq 4)
		\end{equation} 
	\end{lem}
\begin{proof} Compute the traces, for example 
$Trace(Ms_2M²^{-1}s_1)=2-a^2$. \end{proof}
We can now conclude the proof of Proposition \ref{thm sur les random walks}. First we derive from (\ref{quasipositif.link}) that, for any $\beta\in B(3)$, $\chi_4(\sigma_i\beta\sigma_{j}\beta^{-1})=0$, hence $sign(\sigma_i\beta\sigma_{j}\beta^{-1})=0$. \\ Now let $\beta$ be a $3$-braid such that for all the matrix entries $ab$, we have
\begin{equation}\label{B est grand}
|\big({\mathcal B}_{-1}(\beta)\big)_{ab}|>2
\end{equation}
Then, for all $i,j\in\{1,2\}$ $\sigma_i\beta\sigma_{j}\beta^{-1}$ is hyperbolic (Lemma \ref{qui est hyperbolique}). Thus
 $$Meyer({\mathcal B}_{-1}(\sigma_i\beta\sigma_{j}\beta^{-1}),{\mathcal B}_{-1}((\sigma_i\beta\sigma_{j}\beta^{-1})^n)=0$$ (Fact \ref{meyer pour hyperboliques} of \S \ref{section sur le meyer cocycle}). It follows by induction that, for every $n$, $$sign((\sigma_i\beta\sigma_{j}\beta^{-1})^n)=0.$$
 On the other hand, Theorem \ref{prop sur les transients} tells us that the braids not verifying (\ref{B est grand}) are a transient set for the right random $\mu$-walks on $B(3)$. 
	 \qed

\section{The Lissajous toric knots}\label{rappels sur les lissajous toriques}
\subsection{Description} 
\subsubsection{The knots $K(N,q,p)$}\label{paragraphe knot}
In [S-V], the authors exhibited Lissajous toric knots as one of the classes of  boundaries of minimal disks in the $4$-ball with a branch point at the origin. Before that, Lamm investigated these knots in connection with billiards in the solid torus ([L-O]). If $N,q,p$ are integers, with $(N,q)=(N,p)=1$, the $K(N,q,p)$ Lissajous toric knot is defined in a $3D$-cylinder as
$$F_{N,q,p}:[0,2\pi]\longrightarrow\mathbb{S}^1\times\mathbb{R}^2$$
\begin{equation}\label{knot in the cylinder}
F_{N,q,p}:\theta\mapsto(e^{Ni\theta},\sin(q\theta),\cos(p\theta+\alpha))
\end{equation}
for a phase $\alpha$.
Endow $\mathbb{R}^3$ with a coordinate  system $(x,y,z)$ and take the Lissajous curve $C_{p,q}:\theta\mapsto \big(0,2+\sin(q\theta),\cos(p\theta+\alpha)\big)$; then the knot $K(N,q,p)$ is described in $\mathbb{R}^3$ by a point travelling along $C_{p,q}$ while $C_{p,q}$ rotates $N$ times along the vertical axis $Oz$.  We recall a few facts
\begin{fact}\label{no phase}([S-V])
	For a finite number of phases $\alpha$'s the expression in (\ref{knot in the cylinder}) gives us singular crossing points. Otherwise (\ref{knot in the cylinder}) defines a knot and up to mirror transformation, its knot type does not depend on the phase $\alpha$.
\end{fact}
Thus we drop the phase $\alpha$ in (\ref{knot in the cylinder}) and we just talk of a knot $K(N,q,p)$ defined up to mirror symmetry.
\subsubsection{The braids $B(N,q,p)$}\label{paragraphe braid}
 The knot $K(N,q,p)$  has a natural $N$-braid representation $B(N,q,p)$.
\begin{fact}\label{description de la tresse}([S-V]) Let $\tilde{p},\tilde{q},d$ be three positive integers all coprime with $N$ and such that $(\tilde{p},\tilde{q})=1$. Assume (without loss of generality) that $\tilde{q}$ is odd. Then
	$$B(N,d\tilde{q},d\tilde{p})=B(N,\tilde{q},\tilde{p})^d$$
	and there exists a braid $Q_{N,\tilde{q},\tilde{p}}$ such that 
	\begin{equation}\label{braid}
	B(N,\tilde{q},\tilde{p})=Q_{N,\tilde{q},\tilde{p}}\sigma_2^{\epsilon(2)}Q_{N,\tilde{q},\tilde{p}}^{-1}\sigma_1^{\epsilon(1)}
	\end{equation}
	with $\epsilon(1), \epsilon(2)\in\{-1,1\}$.\\
	If $2N$ divides $\tilde{p}+\tilde{q}$ or $\tilde{p}-\tilde{q}$, $\epsilon(1)= \epsilon(2)=1$ so the braid is quasipositive: for $N=3$, this happens if and only if $\tilde{q}$ and $\tilde{p}$ are both odd.
\end{fact}
The braid (\ref{braid}) is a symmetric union as defined by [La] thus: 
\begin{fact}\label{fact ribbon}
	If $N,q,p$ are all mutually prime, $K(N,q,p)$ is a ribbon knot.
\end{fact}

\subsection{The signature of the Lissajous toric knots}
For $N=3$, a Lissajous toric knot has zero signature unless it is a torus knot. More precisely,
\begin{thm}\label{theoreme sur lissajous}
	Let $\tilde{q},\tilde{p},n$ be two positive integers, none of them divisible by $3$ and $\tilde{q},\tilde{p}$ mutually prime.\\
	For every positive integer $n$,
	\begin{enumerate}
		\item If $\tilde{q}$ and $\tilde{p}$ are both odd, up to mirror image, $K(3,n\tilde{q},n\tilde{p})$ is a quasipositive knot with
		\begin{equation}\label{smooth genus} 
		g_4(K(3,n\tilde{q},n\tilde{p}))=d-1
		\end{equation}
		and verifying one of the following
		\begin{enumerate}
			\item 
			the signature of $K(3,n\tilde{q},n\tilde{p})$ is zero
			\item
			$B(3,\tilde{q},\tilde{p})=\sigma_2\sigma_1$ so $K(3,\tilde{q},\tilde{p})$ is a trivial knot and $K(3,n\tilde{q},n\tilde{p})$ is a $(3,n)$-torus knot.
		\end{enumerate}
		\item 
		If $\tilde{q}$ and $\tilde{p}$ have different parities, $K(3,n\tilde{q},n\tilde{p})$ is isotopic to its mirror image, so it has zero signature.  
	\end{enumerate}
	
\end{thm}

REMARK 1. The Lissajous toric knots $K(3,n,n)$ are just the $(3,n)$ torus knots so clearly they are in the case 1 (b) of Theorem \ref{theoreme sur lissajous}, but they are not the only ones. For example if $\tilde{p}=\tilde{q}+6$ the knot $K(3,\tilde{q},\tilde{q}+6)$ (e.g. $K(3,7,13)$) is represented by the braid  $$B(3,\tilde{q},\tilde{q}+6)=\sigma_2\sigma_1$$
thus it is trivial and, for a positive integer $n$, $$B(3,n\tilde{q},n(\tilde{q}+6))=(\sigma_2\sigma_1)^n$$ so $K(3,n\tilde{q},n(\tilde{q}+6))$ is a $(3,n)$-torus knot.\\
In \S \ref{Lissajousstat}, we see that at least for $\tilde{q}$'s up to $100$, the majority of knots of 1. in Theorem \ref{theoreme sur lissajous} 1. are in the case 1. (a) of that Theorem. \\

EXEMPLE. The knot $K(3,5,7)$ is, up to mirror image, the knot $10_{155}$ in the Rolfsen classification ([S-V]) and is represented by the braid \begin{equation} 
B(3,5,7)=\sigma_2\sigma_1^{-1}\sigma_2\sigma_1^{-1}\sigma_2\sigma_1\sigma_2^{-1}\sigma_1\sigma_2^{-1}\sigma_1                 \end{equation} 

\subsection{Proof of Proposition \ref{signature de lissajous}}
The proof depends on the parities of $\tilde{q}$ and $\tilde{p}$; since these numbers are mutually prime, either they are both odd or they have different parities.
\subsubsection{1st case: $\tilde{q}$ and $\tilde{p}$ are both odd}\label{p et q sont impairs}
Up to mirror symmetry, the braid $B(3,\tilde{q},\tilde{p})$ is of the form (Fact \ref{description de la tresse})
 $$B(3,\tilde{q},\tilde{p})=Q_{3,\tilde{q},\tilde{p}}\sigma_2 Q_{3,\tilde{q},\tilde{p}}^{-1}\sigma_1$$ with  ([S-V])\ \ \ \ \ \ $Q_{3,\tilde{q},\tilde{p}}=\sigma_2^{\lambda(1)}\sigma_1^{\lambda(2)}
 \sigma_2^{\lambda(3)}\sigma_1^{\lambda(4)}
 ...
 \sigma_2^{\lambda(\tilde{q}-2)}
 \sigma_1^{\lambda(\tilde{q}-1)}$\\
where $\lambda$ is an expression with values in $\{-1, 1\}$ verifying\\
$\boxed{\lambda(k)=-\lambda(\tilde{q}-k)}$\ \ \ 
so we rewrite $Q_{3,\tilde{q},\tilde{p}}$ and introduce the braid $P$:
\begin{equation}\label{QQ}
Q_{3,\tilde{q},\tilde{p}}=\underbrace{\sigma_2^{\lambda(1)}\sigma_1^{\lambda(2)}\cdot\cdot\cdot
\sigma_2^{\lambda(\frac{\tilde{q}-3}{2})}\sigma_1^{\lambda(\frac{\tilde{q}-1}{2})}}_{P}
\sigma_2^{-\lambda(\frac{\tilde{q}-1}{2})}
\sigma_1^{-\lambda(\frac{\tilde{q}-3}{2})}
...\sigma_2^{-\lambda(2)}\sigma_1^{-\lambda(1)}
\end{equation}
We take the images of these braids under ${ \mathcal B}_{-1}$ and let $${\mathcal Q}={\mathcal B}_{-1}(Q_{3,\tilde{q},\tilde{p}})\ \ \ \ \ \ \ \ \ \ \ \
{\mathcal P}={\mathcal B}_{-1}(P).$$
We notice that the $s_i$'s of (\ref{les generateurs}) verify
$\boxed{s_2={}^{t}s_1^{-1}}$
and derive
\begin{equation}
{\mathcal Q}={\mathcal P}^t{\mathcal P}
\end{equation}
	Let ${\mathcal P}=\begin{pmatrix}
	a & b \\
	c & d 
	\end{pmatrix}$. We compute\\ 
	$Trace({\mathcal B}_{-1}(B(3,\tilde{q},\tilde{p})))=Trace[{\mathcal P}^t{\mathcal P}s_2^t{\mathcal P}^{-1}{\mathcal P}^{-1}s_1]=2-(a^2+b^2)^2$. Thus
	\begin{lem}
	The matrix ${\mathcal B}_{-1}(B(3,\tilde{q},\tilde{p}))$ is hyperbolic except if $a,b\in\{-1,0,1\}$.
	\end{lem}
To go from the matrices back to braids, we recall
\begin{prop} Let $\Delta=\sigma_1\sigma_2\sigma_1\sigma_2\sigma_1\sigma_2$. It is a pure braid which belongs to the center of $B(3)$. If $\beta_1,\beta_2$ are two $3$-braids with ${\mathcal B}_{-1}(\beta_1)={\mathcal B}_{-1}(\beta_2)$. Then for some $k$, $\beta_2=\Delta^k\beta_1$.
	\end{prop}
If ${\mathcal P}$ is hyperbolic, it
has one of the following forms, for some integer $h$.
\begin{itemize}
	\item 
	${\mathcal P}=(\pm 1)\begin{pmatrix}
	1 & 0 \\
	h & 1 
	\end{pmatrix}=(\pm 1)s_1^{-h}$\\
	Thus ${\mathcal Q}={\mathcal P}^t{\mathcal P}=s_1^{-h}s_2^{h}$ and, for some integer $k$,
	$$\beta=(\sigma_1^{-h}\sigma_2^{h}\Delta^k)\sigma_2(\Delta^{-k}\sigma_2^{-h}\sigma_1^{h})\sigma_1
	=\sigma_1^{-h}(\sigma_2\sigma_1)\sigma_1^{h}
	$$
	so $\beta$ is conjugate to $\sigma_2\sigma_1$ and the knot is as in 1. (b) of Theorem \ref{theoreme sur lissajous}.
	\item 
	${\mathcal P}=(\pm 1)\begin{pmatrix}
	0 & 1 \\
	-1 & h 
	\end{pmatrix}=(\pm 1)s_1^{1-h}s_2s_1$\\
	Since $\sigma_1\sigma_2\sigma_1=\sigma_2\sigma_1\sigma_2$, we have ${\mathcal Q}=\Delta^ks_1^{-h}s_2^{h}$ and similarly to above,
	$$\beta=\sigma_1^{-h}(\sigma_2\sigma_1)\sigma_1^{h}.$$
	Again we are in Case 1. (b) of Theorem \ref{theoreme sur lissajous}.
	\item 
	${\mathcal P}=(\pm 1)\begin{pmatrix}
	1 & 1 \\
	h & h+1 
	\end{pmatrix}=(\pm 1)s_1^{-h}s_2$\\
	Then ${\mathcal Q}=\Delta^ks_1^{-h}s_2s_1^{-1}s_2^{h}$ and
	$\beta=\sigma_1^{-h}(\sigma_2\sigma_1^{-1}\sigma_2\sigma_1\sigma_2^{-1}\sigma_1)\sigma_1^{-h}$
	This braid $\beta$ closes into a link with $3$ components, thus it cannot be the braid of a knot.
	
	\item 
	${\mathcal P}=(\pm 1)\begin{pmatrix}
	1 & -1 \\
	h & 1-h
	\end{pmatrix}=s_1^{-h}s_2^{-1}$ and 
	$\beta=\sigma_1^{-h}(\sigma_2^{-1}\sigma_1\sigma_2)(\sigma_1^{-1}\sigma_2\sigma_1)\sigma_1^{h}$.
	Again $\beta$ closes in a link with $3$ components. \qed
\end{itemize}

\subsubsection{2nd case: $\tilde{q}$ and $\tilde{p}$ have different parities}
We assume that $\tilde{q}$ is odd and $\tilde{p}$ is even.
The knot $K(N,d\tilde{q},d\tilde{p})$ is defined by  
	 the function $F_{N,d\tilde{q},d\tilde{p}}$ (see (\ref{knot in the cylinder}) above), so we can also define it by the function
	 $$\theta\mapsto F_{N,d\tilde{q},d\tilde{p}}(\theta+\frac{\pi}{d})=
	 (e^{iN\theta}e^{iN\frac{\pi}{d}},-\sin(d\tilde{q}\theta),\cos(d\tilde{p}\theta+\alpha))$$
	 thus $K(N,d\tilde{q},d\tilde{p})$ is invariant under the transformation
	 \begin{equation}\label{grosse matrice}
	 \begin{pmatrix}
	 \cos(\frac{N\pi}{d}) & -\sin(\frac{N\pi}{d}) & 0 & 0\\
	 \sin(\frac{N\pi}{d}) & \cos(\frac{N\pi}{d}) & 0 & 0\\
	 0 & 0 & -1 & 0\\ 
	 0 & 0 & 0 & 1
	 \end{pmatrix}
	 \end{equation}
	 which reverses the orientation on $\mathbb{R}^4$ and on $\mathbb{S}^3$. Thus thus the knot is isotopic to its mirror image and  has zero signature. \qed

	 \section{Statistics}
	 The following estimates were done with SageMath.
	 \subsection{Random quasipositive braids}
	  We define $Z=\{\beta\in B(3): |\big({\mathcal B}_{-1}(\beta)\big)_{11}|> 2\}.$	\\  It follows from the previous discussion that, if a $3$-braid $\beta$ belongs to $Z$, then for all positive integers $n$, we have  $$sign\Big(\widehat{(\beta\sigma_1\beta^{-1}\sigma_2)^n}\Big)=0.$$
	 Let $W$ be the $\mu$-random walk on $B(3)$  given by the probability $\mu$ on $B(3)$ where
	 $$\mu(\sigma_1)= \mu(\sigma_2)=\mu(\sigma_1^{-1})= 
	 \mu(\sigma_2^{-1})=\frac{1}{4}$$
Here are the probabilities $p(n)$  that $W$ hits $Z$ at the $n$-th step, for $n\leq 12$.  
	 
	 \[ \left| \begin{array}
	 { c | c | c | c | c | c | c | c| c | c | c |c|c|c}	
	
	 \hline
	n& 1 &  2 & 3  & 4 & 5 & 6 & 7 & 8 & 9 & 10 & 11&12&13\\
	\hline 
	p(n) & 0 &  0 & 0.06  & 0.11 & 0.17 & 0.22 & 0.27 & 0.32 & 0.36 & 0.41 & 0.45&0.48&0.52\\
	 \hline
	 \end{array} \right|.\]
	 
	 \subsection{Colorability}
	 The following table gives the probability to be $p$-colorable ($p=3,5,7$) for the closure of a braid given by $N$ steps of a $\mu$-random walk ($N=3,...,10$).
	 
	 \[ \left| \begin{array}
	 { c | c | c | c | c | c | c | c| c | c | c |c|c|c}	
	 
	 \hline
	 N&\dfrac{p}{p^2-1} &N=3 & N= 4 & N=5  & N=6 & N=7 & N=8 & N=9 & N=10 \\
	 \hline 
	 p=3&0.375 & 0.19 & 0.17& 0.24 &0.25  &0.28  &0.29  &0.31  &0.32  \\
	 \hline
	 p=5 &0.21 &0.125&  0.05& 0.16 &0.1  &0.17  &0.14  &0.18  &0.16  \\
	 \hline
	 p=7 &0.15 &0.125 &0.03  &0.15  &0.07  & 0.16 &0.1  &0.15  &0.12  \\
	 \hline
	 \end{array} \right|.\]
	 
\subsection{Lissajous toric knots with $3$ strands}\label{Lissajousstat}
We know from  [S-V]	that the type if a knot $K(N,q,p)$ only depends, up to mirror symmetry, on the congruence of $P$ modulo $q$. Thus for a given $q$, we let $p(q)$ be the cardinal number of 
\begin{equation}
\{p\in\mathbb{N}:q<p<2q, (p,q)=(p,6)=1, (q,p)\ \mbox{verifies 1. (b) of Theorem}\ \ref{theoreme sur lissajous}\}
\end{equation}
We compute $p(q)$, for $3<p<200$, $(q,6)=1$ and $q\leq 1(4)$ (this last assumption is just to make computations simpler).
	 \[ \left| \begin{array}
	{ c | c | c | c| c | c | c |c | c | c | c | c| c | c | c |c | c | c | c | c| c | c | c |c|c|c }	
	\hline
	q& 	5 & 13 & 17& 25 & 29 & 37 & 41 & 49 & 53 & 61 & 65 & 73 & 77 & 85 & 89&97&101&109&113\\
	 	\hline
	 \%	&100	 &50  &80  &57 &78  &67  &85  &71  &82  &80  &81  &75  &85  &73 &79 &75&85&78&89  \\ 
	 	\hline
	 \end{array} \right.\]
	  \[ \left| \begin{array}
	 { c | c | c | c| c | c | c |c | c | c | c | c| c | c | c |c | c | c | c | c| c | c | c |c }	
	 \hline
	q  & 121&125 & 133 & 137 & 145 & 149 & 157 & 161 & 169& 173 & 181 &185 & 193 & 197\\
	 \hline
	\%	& 84& 85&86 & 84 & 74 & 84 & 85 & 89&83&84&78&88&88&88\\ 
	 \hline
	 \end{array} \right.\]
	 
\footnotesize{Marc Soret: Universit\'e de Tours, D\'ep. de Math\'ematiques, 37000 Tours, France,\\
	marc.soret@univ-tours.fr\\
	 \\
Marina Ville: Univ Paris Est Creteil, CNRS, LAMA, F-94010 Creteil, France\\	villemarina@yahoo.fr}

\begin{thebibliography}{ru}
	 	\bibitem[Ac]{ac} J. Achter, {\it The distribution of class groups of functions} Journal of Pure and Applied Algebra 204 (2006) 316 – 333
	 	\bibitem[A-H]{ah} J. Achter, J. Holden {\it Notes on an analogue of the Fontaine-Mazur
	 	conjecture}, Journal de Th\'eorie des Nombres de Bordeaux 15 (2003), 627–637
	 	\bibitem[A'C]{a'c} N. A'Campo, {\it Tresses, monodromie et le groupe symplectique}, Comm. Math. Helvetici (54) 
	 	318-327, (1979)
	 	 \bibitem[As]{as} M. Aschbacher, {\it Finite group theory}, Cambridge Univ. Press, London 1986
	 	 \bibitem[B-K]{bk} M. Brandenbursky, J. Kedra {\it Concordance group and stable commutator length in braid groups}, Algebraic \& Geometric Topology (5), 2861-2886, (2015).
	 	 \bibitem[Br]{br} M. Brandenbursky, {\it On quasi-morphisms from knot and braid invariants}, Jour. of Knot Theory and its Ramifications, vol. 20, No. 10 (2011), 1397-1417.
	 	 \bibitem[Di]{di} P. Diaconis, {\it Representations in probability and statistics}, IMS, Hayward, 1986
	 	 \bibitem[G-G]{rr} J.-M. Gambaudo, E. Ghys {\it Braids and signatures}, Bull. SMF 133 (2005), 541-579
	 	 \bibitem[It]{it} T. Ito, {\it On a structure of random open books and closed braids}, Proc. Japan Acad. 91, Ser. A (2015) 
 	\bibitem[K-T]{kt} C. Kassel, V. Turaev, {\it Braid groups}, Graduate texts in maths, Springer-Verlag New York (2008)
 %	\bibitem[Ko]{ko} P. Kohn, {\it Unlocking two component links}, Osaka Jour. Math. 30 (1993) 741-752
 \bibitem[La]{la} C. Lamm, {\it Symmetric unions and ribbons knots}, Osaka  J. Math. 37  (2000),  537-550
 		\bibitem[L-O]{la0} C.  Lamm, D. Obermeyer {\it Billiard knots in a cylinder} J. Knot Theory and its Ramifications 8(3) (1999) 353-366.
 		\bibitem[L-R]{lar} G. Lachaud, R. Rolland {\it On the number of points of algebraic sets over finite fields}, Journal of Pure and Applied Algebra
 		219 (11) ,  5117-5136 (2015)
 		\bibitem[Ma]{ma} J. Ma, {\it The closure of a random braid is a hyperbolic link}, Proc. AMS 142(2) (2014) 695–701.
 	%	On the number of points of algebraic sets over finite fields
 		\bibitem[Mal]{ma} A. V. Malyutin, {\it Quasimorphisms, random walks, and transient subsets in countable groups}, 	
 		Jour. of Math. Sciences 181(6) 871--885  (2012)
 			\bibitem[Mi]{mi} J.S. Milne, {\it Algebraic groups}, Cambridge Studies in Advanced Mathematics (170), CUP 2017
 		\bibitem[Mu]{mu}	K. Murasugi, {\it On a certain numerical invariant of link types}, Trans. Amer. Math.
 			Soc. 117 (1965), 387-422. 
 			\bibitem[O'M]{om} O. T. O'Meara, {\it Symplectic groups}, AMS, Providence (1978)
 				\bibitem[Pr]{pr} J. H. Przytycki {\it Positive knots have negative signature}, Bull. Ac. Pol.: Math. 37, 1989, 559-562.
 		\bibitem[Ri]{ri2} I. Rivin, {\it Generic phenomena in groups} in {\it Thin groups and superstrong approximation} MSRI Publications, Volume 61 (2013) 
 		\bibitem[Ro]{ro} D. Rolfsen, {\it Knots and links}, AMS Chelsea Publishing (2003) 
 \bibitem[Ru]{ru} L. Rudolph, {\it Quasipositivity as an obstruction to sliceness}, Bull. Amer. Math. Soc. 29 (1993), 51-59.
 	\bibitem[S-V]{rv} M. Soret, M. Ville, {\it Lissajous-toric knots},  Jour. of Knot Theory and Its Ramifications 29(1) (2020) 
 	\bibitem[Ta]{rt} T. Tanaka, {\it Unknotting numbers  of quasipositive knots}, Topology and its applications 88 (1998) 239-246
 	\bibitem[wi]{wi} https://groupprops.subwiki.org/wiki/Order$\_$formulas$\_$for$\_$symplectic$\_$groups
 \end{thebibliography}
\end{document}